\documentclass[reqno,12pt]{amsart}
\usepackage{eurosym}
\usepackage{amsfonts}
\usepackage{amssymb}
\usepackage{graphicx}
\usepackage{pstricks}
\usepackage{amsmath}
\usepackage{amsmath}
\usepackage{amsxtra}

\usepackage{cite}

\usepackage{mathrsfs, amsthm, xifthen, verbatim, mathrsfs}
\usepackage[normalem]{ulem}

\usepackage{moreverb}

\newrgbcolor{lightorange}{1 .5 0}
\newrgbcolor{lightgreen}{0.8 1 .8}
\newrgbcolor{lightyellow}{1 1 0}
\newrgbcolor{lightblue}{0 1 1}
\definecolor{darkgreen}{rgb}{.1,.5,0}

\theoremstyle{plain}
\newtheorem{theorem}{Theorem}[section]
\newtheorem{lemma}[theorem]{Lemma}

\newtheorem{definition}[theorem]{Definition}

\newtheorem{problem}[theorem]{Problem}

\theoremstyle{definition}

\newtheorem{observation}[theorem]{Observation}
\newtheorem{remark}[theorem]{Remark}

\numberwithin{equation}{section} 
\def\({\left(}
\def\){\right)}

\def\Z{\mathbb{Z}_+^2}
\def\W{W_{(\alpha,\beta)}}

\def\lk{\mathcal{L}_{\mathbf{k}}}
\def\bk{\mathbf{k}}

\def\shift{\operatorname{shift}}
\def\max{\operatorname{max}}

\def\stc{\operatorname{\mathcal{S}}}
\def\fnd{\operatorname{\mathcal{F}}}

\newcommand\ddfrac[2]{\frac{\displaystyle #1}{\displaystyle #2}}

\begin{document}
\title[Reconstruction-of-the-Measure Problem for Canonical Invariant Subspaces]{%
Solution of the Reconstruction-of-the-Measure Problem \\ for Canonical Invariant Subspaces}
\author{Ra\'{u}l E. Curto}
\address{Department of Mathematics, The University of Iowa, Iowa City, Iowa
52242}
\email{raul-curto@uiowa.edu}
\urladdr{http://www.math.uiowa.edu/\symbol{126}rcurto/}
\author{Sang Hoon Lee}
\address{Department of Mathematics, Chungnam National University, Daejeon,
34134, Republic of Korea}
\email{slee@cnu.ac.kr}
\urladdr{}
\author{Jasang Yoon}
\address{School of Mathematical and Statistical Sciences, The University of
Texas Rio Grande Valley, Edinburg, Texas 78539, USA}
\email{jasang.yoon@utrgv.edu}

\begin{abstract}
We study the Reconstruction-of-the-Measure Problem (ROMP) for commuting $2$-variable weighted shifts $\W$, when the initial data are given as the Berger measure of the restriction of $\W$ to a canonical invariant subspace, together with the marginal measures for the $0$--th row and $0$--th column in the weight diagram for $\W$. \ We prove that the natural necessary conditions are indeed sufficient. \ When the initial data correspond to a soluble problem, we give a concrete formula for the Berger measure of $\W$. \ Our strategy is to build on previous results for back-step extensions and one-step extensions. \ A key new theorem allows us to solve ROMP for two-step extensions. \ This, in turn, leads to a solution of ROMP for arbitrary canonical invariant subspaces of $\ell^2(\Z)$. 
\end{abstract}

\thanks{The second author of this paper was partially supported by NRF
(Korea) grant No. 2020R1A2C1A0100584611.}
\thanks{The third named author was partially supported by a grant from the
University of Texas System and the Consejo Nacional de Ciencia y Tecnolog%
\'{\i}a de M\'{e}xico (CONACYT)}
\maketitle
\tableofcontents

\setcounter{tocdepth}{2}

\section{\label{Int}Introduction}

Let $\mathcal{H}$ be a complex Hilbert space and let $\mathcal{B}(\mathcal{H}%
)$ denote the algebra of bounded linear operators on $\mathcal{H}$. \ We say
that $T\in \mathcal{B}(\mathcal{H})$ is \emph{normal} if $T^{\ast
}T=TT^{\ast }$, \emph{subnormal} if $T=N |_{\mathcal{H}}$, where $N$ is
normal and $N(\mathcal{H})\subseteq \mathcal{H}$, and hyponormal if $T^{\ast }T \ge TT^{\ast }$. \ These notions extend to commuting $n$-tuples of Hilbert space operators $\mathbf{T} \equiv (T_1,\ldots,T_n)$; for instance, $\mathbf{T}$ is normal if $T_i^{\ast }T_j=T_jT_i^{\ast }$ for all $i,j=1,\ldots,n$, and $\mathbf{T}$ is subnormal if $\mathbf{T} = \mathbf{N} |_{\mathcal{H}}$, where $\mathbf{N}$ is normal on a larger Hilbert space $\mathcal{K}$ and $N_i\mathcal{H} \subseteq \mathcal{H}$ for all $i=1,\ldots,n$. \ 

In this paper we will focus on the case of $\mathcal{H}=\ell^2(\Z)$ and $\mathbf{T}$ a $2$-variable weighted shift $\W \equiv(T_1,T_2)$. \ As is well known, the subnormality of such pairs is characterized by the existence of a representing measure (known as the Berger measure of $\W$) for the family of moments of the pair of double-indexed sequences $(\alpha,\beta)$. \ We will solve the so-called Reconstruction-of-the-Measure Problem (ROMP) when $\W$ admits partial Berger measures; more specifically, when the restriction of $\W$ to a canonical invariant subspace of $\ell^2(\Z)$ is subnormal. \ That restriction generates a finite family $\mathcal{F}$ of (localized) Berger measures $\mu_1,\ldots,\mu_p$ which satisfy natural compatibility conditions. \ Two other $1$-variable Berger measures, $\sigma$ and $\tau$ (corresponding to the $0$--th row and $0$--th column in the weight diagram for $\W$), are assumed to be part of the initial data. \ ROMP asks for necessary and sufficient conditions on the members of $\mathcal{F}$ for the subnormality of $\W$. \ When soluble, one must also reconstruct the Berger measure of $\W$ from $\mathcal{F}$ and the marginal measures $\sigma$ and $\tau$. \ Our main result (Theorem \ref{canonicalthm}) gives a complete solution of ROMP in the case of canonical invariant subspaces. 

\medskip
Besides their relevance for the construction of examples and counterexamples
in Hilbert space operator theory, weighted shifts can also be used to detect
properties such as subnormality, via the Lambert-Lubin Criterion (\cite{Lamb}%
, \cite{Lub}): A commuting pair $(T_{1},T_{2})$ of injective operators
acting on a Hilbert space $\mathcal{H}$ admits a commuting normal extension
if and only if for every nonzero vector $x \in \mathcal{H}$, the $2$-variable
weighted shift with weights $\alpha_{(i,j)} :=\frac{\left\|T_1^{i+1}T_2^jx
\right\|}{\left\|T_1^iT_2^jx\right\|}$ and $\beta_{(i,j)} :=\frac{%
\left\|T_1^iT_2^{j+1}x\right\|}{\left\|T_1^iT_2^jx\right\|}$ has a normal
extension.

For $\omega \equiv \{\omega _{n}\}_{n=0}^{\infty }\in \ell ^{\infty }(%
\mathbb{Z}_{+})$ a bounded sequence of positive real numbers (called \emph{weights}), let $W_{\omega }:\ell ^{2}(\mathbb{Z}_{+})\rightarrow \ell ^{2}(%
\mathbb{Z}_{+})$ be the associated unilateral weighted shift, defined by $%
W_{\omega }e_{n}:=\omega _{n}e_{n+1}\;($all $n\geq 0)$, where $%
\{e_{n}\}_{n=0}^{\infty }$ is the canonical orthonormal basis in $\ell ^{2}(%
\mathbb{Z}_{+}).$ \ The \emph{moments} of $\omega $ are given as
\begin{equation*}
\gamma _{k}\equiv \gamma _{k}(\omega ):=\left\{
\begin{tabular}{ll}
$1$, & $\textrm{if }k=0$ \\
$\omega _{0}^{2}\cdots \omega _{k-1}^{2}$, & $\text{if }k>0.$%
\end{tabular}%
\right.
\end{equation*}%
It is easy to see that $W_{\omega }\equiv \shift (\omega _{0},\omega
_{1},\cdots )$ is never normal, and that it is hyponormal if and only if $%
\omega _{0}\leq \omega _{1}\leq \cdots $. \ We recall a well known
characterization of subnormality for single variable weighted shifts, due to
C. Berger (cf. \cite[II.6.10]{Con}) and independently established
by R. Gellar and L.J. Wallen \cite{GeWa}: \ $W_{\omega }$ is subnormal if and
only if there exists a probability measure $\sigma $ supported in $%
[0,\left\Vert W_{\omega }\right\Vert ^{2}]$ (called the \emph{Berger
measure }of $W_{\omega }$) such that for $k\geq 1$
\begin{equation}
\gamma _{k}(\omega )\equiv \gamma _{k}(\sigma )=\int t^{k}d\sigma (t).
\label{measure}
\end{equation}

Similarly, consider double-indexed positive bounded sequences $\alpha \equiv
\{\alpha _{\left( k_{1},k_{2}\right) }\},\beta \equiv \{\beta _{\left(
k_{1},k_{2}\right) }\}\in \ell ^{\infty }(\mathbb{Z}_{+}^{2})$, $%
(k_{1},k_{2})\in \mathbb{Z}_{+}^{2}:=\mathbb{Z}_{+}\times \mathbb{Z}_{+}$
and let $\ell ^{2}(\mathbb{Z}_{+}^{2})$\ be the Hilbert space of
square-summable complex sequences indexed by $\mathbb{Z}_{+}^{2}$; in a
canonical way, $\ell ^{2}(\mathbb{Z}_{+}^{2})$ is isometrically isomorphic
to $\ell ^{2}(\mathbb{Z}_{+})\bigotimes \ell ^{2}(\mathbb{Z}_{+})$. \ We
define the $2$-variable weighted shift $\W \equiv (T_{1},T_{2})$ by
\begin{equation*}
T_{1}e_{\left( k_{1},k_{2}\right) }:=\alpha _{\left( k_{1},k_{2}\right)
}e_{\left( k_{1},k_{2}\right) \mathbf{+}\varepsilon _{1}}
\end{equation*}%
\begin{equation*}
T_{2}e_{\left( k_{1},k_{2}\right) }:=\beta _{\left( k_{1},k_{2}\right)
}e_{\left( k_{1},k_{2}\right) \mathbf{+}\varepsilon _{2}},
\end{equation*}%
where $\mathbf{\varepsilon }_{1}:=(1,0)$ and $\mathbf{\varepsilon }%
_{2}:=(0,1)$ (see Figure \ref{Figure 1}). \ Clearly, $T_{1}$ commutes with $%
T_{2}$ if and only if
\begin{equation}
\beta _{\left( k_{1},k_{2}\right) \mathbf{+}\varepsilon _{1}}\alpha _{\left(
k_{1},k_{2}\right) }=\alpha _{\left( k_{1},k_{2}\right) \mathbf{+}%
\varepsilon _{2}}\beta _{\left( k_{1},k_{2}\right) }\;\;\left( \textrm{all }%
\left( k_{1},k_{2}\right) \in \mathbb{Z}_{+}^{2}\right) .  \label{commuting}
\end{equation}%
In an entirely similar way one can define multivariable weighted shifts.


\setlength{\unitlength}{1mm} \psset{unit=1mm}
\begin{figure}[th]
\begin{center}
\begin{picture}(155,65)

\rput(0,23){

\psline{->}(20,20)(69,20)
\psline(20,40)(66,40)
\psline(20,60)(66,60)
\psline{->}(20,20)(20,70)
\psline(40,20)(40,67)
\psline(60,20)(60,67)

\put(12,16){\footnotesize{$(0,0)$}}
\put(37,16){\footnotesize{$(1,0)$}}
\put(57,16){\footnotesize{$(2,0)$}}

\put(28,21){\footnotesize{$\alpha_{00}$}}
\put(48,21){\footnotesize{$\alpha_{10}$}}
\put(61,21){\footnotesize{$\alpha_{20}$}}

\put(28,41){\footnotesize{$\alpha_{01}$}}
\put(48,41){\footnotesize{$\alpha_{11}$}}
\put(61,41){\footnotesize{$\alpha_{21}$}}

\put(28,61){\footnotesize{$\alpha_{02}$}}
\put(48,61){\footnotesize{$\alpha_{12}$}}
\put(61,61){\footnotesize{$\alpha_{22}$}}

\psline{->}(35,14)(50,14)
\put(42,10){$\rm{T}_1$}
\psline{->}(10,35)(10,50)
\put(4,42){$\rm{T}_2$}

\put(11,40){\footnotesize{$(0,1)$}}
\put(11,60){\footnotesize{$(0,2)$}}

\put(20,30){\footnotesize{$\beta_{00}$}}
\put(20,50){\footnotesize{$\beta_{01}$}}
\put(20,65){\footnotesize{$\beta_{02}$}}

\put(40,30){\footnotesize{$\beta_{10}$}}
\put(40,50){\footnotesize{$\beta_{11}$}}
\put(40,65){\footnotesize{$\beta_{12}$}}

\put(60,30){\footnotesize{$\beta_{20}$}}
\put(60,50){\footnotesize{$\beta_{21}$}}

}


\rput(82,30){

\psline{->}(0,0)(60,0)
\psline(0,10)(60,10)
\psline(0,20)(60,20)
\psline(0,30)(60,30)
\psline(0,40)(60,40)
\psline(0,0)(0,40)
\psline(10,0)(10,40)
\psline(20,0)(20,40)
\psline(30,0)(30,40)
\psline(40,0)(40,40)
\psline(50,0)(50,40)
\psline(60,0)(60,40)

\psline[linecolor=blue, linewidth=2.4pt](0,0)(10,0)
\psline[linecolor=blue, linewidth=2.4pt](10,0)(10,20)
\psline[linecolor=blue, linewidth=2.4pt](10,20)(40,20)
\psline[linecolor=blue, linewidth=2.4pt](40,20)(40,30)
\psline[linecolor=blue, linewidth=2.4pt](40,30)(50,30)
\psline[linecolor=blue, linewidth=2.4pt](50,30)(50,40)
\psline[linecolor=blue, linewidth=2.4pt](50,40)(60,40)

\put(56,43){\footnotesize{$(k_{1},k_{2})$}}
\put(-4,-4){\footnotesize{$(0,0)$}}

}

\end{picture}
\end{center}
\caption{(left) \ Weight diagram of a generic $2$-variable weighted shift $\W \equiv (T_{1},T_{2})$; (right) \ nondecreasing path used to compute the moment $\gamma_{(k_1,k_2)}$ .}
\label{Figure 1}
\end{figure}


For a commuting $2$-variable weighted shift $\W$, and for $k_1,k_2 \mathbb{Z}_+$, the \emph{moment} of $(\alpha ,\beta )$ of order $\left( k_{1},k_{2}\right) $ is defined by
\begin{equation}
\gamma _{\left( k_{1},k_{2}\right) }(\alpha ,\beta )\smallskip :=\left\{
\begin{tabular}{ll}
$1$, & $\textrm{if }\left( k_{1},k_{2}\right) =\left( 0,0\right) $ \\
$\alpha _{(0,0)}^{2}\cdots \alpha _{(k_{1}-1,0)}^{2}$, & $\textrm{if }%
k_{1}\geq 1\textrm{ and }k_{2}=0$ \\
$\beta _{(0,0)}^{2}\cdots \beta _{(0,k_{2}-1)}^{2}$, & $\textrm{if }k_{1}=0%
\textrm{ and }k_{2}\geq 1$ \\
$\alpha _{(0,0)}^{2}\cdots \alpha _{(k_{1}-1,0)}^{2}\cdot \beta
_{(k_{1},0)}^{2}\cdots \beta _{(k_{1},k_{2}-1)}^{2}$, & $\textrm{if }k_{1}\geq
1\textrm{ and }k_{2}\geq 1$.%
\end{tabular}
\right. 
\label{gamma}
\end{equation}%
We remark that, due to the commutativity condition (\ref{commuting}), $%
\gamma _{\left( k_{1},k_{2}\right) }\equiv \gamma _{\left(
k_{1},k_{2}\right) }(\alpha ,\beta )$ can be computed using any
nondecreasing path from $(0,0)$ to $(k_{1},k_{2})$ (see Figure \ref{Figure 1}(right)).

We also recall a well known characterization of subnormality for
multivariable weighted shifts $\mathbf{T\equiv (}T_{1},\cdots ,T_{n})$ \cite%
{JeLu}; for
simplicity, we state it in the case $n=2$: $\W \equiv
(T_{1},T_{2})$ is subnormal if and only if there is a probability measure $%
\mu $ defined on the rectangle $R=[0,a_{1}]\times \lbrack 0,a_{2}]$, where $%
a_{i}:=\left\Vert T_{i}\right\Vert ^{2}$, such that%
\begin{equation*}
\gamma _{\left( k_{1},k_{2}\right) }=\int_{R}s^{k_{1}}t^{k_{2}}d\mu (s,t),
\end{equation*}%
for all $\left( k_{1},k_{2}\right) $ ${\mathbf{\in \mathbb{Z}}_{+}^{2}}$.

\medskip


\section{Statement of Main Results}

For $\bk \in \Z$, let $\lk:= \bigvee \{e_{\bk + \mathbf{p}}: \mathbf{p} \in \Z \} \subseteq \ell^2(\Z)$. \ Of special significance in this paper is the closed subspace $\mathcal{L}_{(k,\ell)}:= \mathcal{L}_{(0,\ell)} \bigcap \mathcal{L}_{(k,0)}$; a visual representation of $\mathcal{L}_{(k,\ell)}$ appears in Figure \ref{Figure 2}. \ \ A closed subspace $\mathcal{R}$ of $\ell^2(\Z)$ is called canonical if $\mathcal{R}=\bigvee_{\{\bk \in \Z: e_{\bk} \in \mathcal{R}\}} \mathcal{L}_{\bk}$.

\setlength{\unitlength}{1mm} \psset{unit=1mm} 
\begin{figure}[ht]
\begin{center}
\begin{picture}(105,80)

\rput(0,25){

\pspolygon*[linecolor=lightgreen](40,40)(40,83)(83,83)(83,40)
\pspolygon*[linecolor=lightyellow](40,20)(40,40)(83,40)(83,20)
\pspolygon*[linecolor=lightblue](20,40)(40,40)(40,83)(20,83)

\psline{->}(20,20)(88,20)
\psline(20,40)(85,40)
\psline(20,60)(85,60)
\psline(20,80)(85,80)
\psline{->}(20,20)(20,88)
\psline(40,20)(40,85)
\psline(60,20)(60,85)
\psline(80,20)(80,85)

\put(11,16){\footnotesize{$(0,0)$}}
\put(35,16){\footnotesize{$(k,0)$}}
\put(55,16){\footnotesize{$(k+1,0)$}}
\put(78,16){\footnotesize{$(k+2,0)$}}

\put(27,21){\footnotesize{$\cdots$}}
\put(87,21){\footnotesize{$\cdots$}}
\put(27,81){\footnotesize{$\cdots$}}

\put(11,38){\footnotesize{$(0,\ell)$}}
\put(7,58){\footnotesize{$(0,\ell+1)$}}
\put(7,78){\footnotesize{$(0,\ell+2)$}}

\put(21.5,28){\footnotesize{$\vdots$}}
\put(21.5,84){\footnotesize{$\vdots$}}
\put(82,28){\footnotesize{$\vdots$}}

\put(65,34){\footnotesize{$\mathcal{L}_{(k,0)}$}}
\put(30,68){\footnotesize{$\mathcal{L}_{(0,\ell)}$}}
\put(49,62){\footnotesize{{{\color{blue} $\mathcal{L}_{(0,\ell)}\bigcap \mathcal{L}_{(k,0)}$}}}}

}

\end{picture}
\end{center}
\caption{The subspaces $\mathcal{L}_{(0,\ell)}$, $\mathcal{L}_{(k,0)}$ and {\color{blue} $\mathcal{L}_{(k,\ell)}:=\mathcal L_{(0,\ell)} \bigcap \mathcal{L}_{(k,0)}$}} 
\label{Figure 2}
\end{figure}
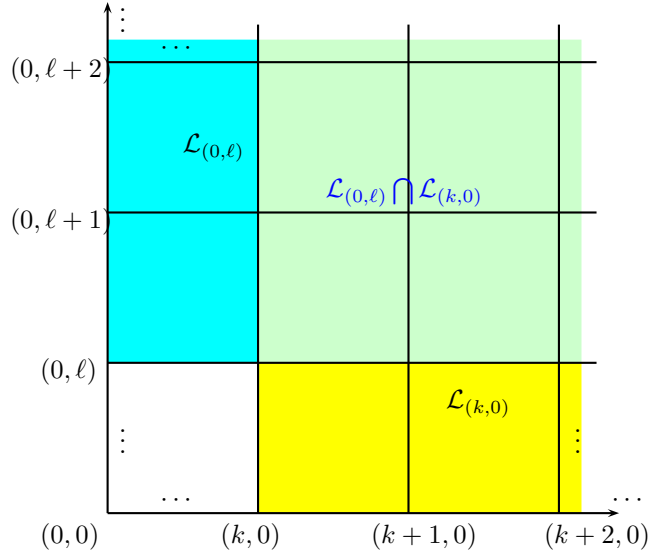

Our first main result is Theorem \ref{thmgeneralized}, which extends the main result in \cite{CLY5} to the case of arbitrary subspaces $\mathcal{L}_{(0,\ell)}$ and $\mathcal{L}_{(k,0)}$. \ That is, given Berger measures $\mu_{(0,\ell)}$ and $\mu_{(k,0)}$ satisfying the compatibility condition
\begin{equation} \label{compatnew}
\gamma_{(0,\ell)} s^k d \mu_{(0,\ell)}(s,t) = \gamma_{(k,0)} t^{\ell} d \mu_{(k,0)}(s,t),
\end{equation}
then the subnormality of $\W$ is guaranteed once the natural necessary conditions given below hold:
\begin{itemize}
\item[(i)] \ $\frac{1}{t^{\ell}} \in L^1(\mu_{(0,\ell)})$;
\item[(ii)] \ $\frac{1}{s^{k}} \in L^1(\mu_{(k,0)})$;
\item[(iii)] \ $\gamma_{(k,0)} \left\|\frac{1}{s^k}\right\|_{\mu_{(k,0)}} = \lambda \gamma_{(0,\ell)} \left\|\frac{1}{s^k}\right\|_{\mu_{(k,0)}} \le 1$; and
\item[(iv)] \ $\gamma_{(0,\ell)} \{ \int_Y \frac{d \; \mu_{(0,\ell)}(s,t)}{t^{\ell}}+\lambda \left\|\frac{1}{s^k}\right\|_{\mu_{(k,0)}} d \delta_0(s)-\frac{\lambda}{s^k} \int_Y d\mu_{(k,0)}(s,t) \} \le d \delta_0(s)$.
\end{itemize}

Next, let $\W$ be a $2$-variable weighted shift such that (i) $\W |_{\mathcal{L}_{(k,\ell)}}$ is subnormal with $2$--variable Berger measure $\nu$, (ii) $\W$ restricted to the $0$--th row is subnormal with $1$--variable Berger measure $\sigma$, and (iii) $\W$ restricted to the $0$--column is subnormal with $1$--variable Berger measure $\tau$. \ 

Our second main result is Theorem \ref{bigthm}: given the initial data $\nu$, $\sigma$ and $\tau$, the natural necessary conditions for the subnormality of $\W$ are also sufficient. \ The proof uses a new idea together with appropriate generalizations of results in \cite{CuYo1,CuYo2,ROMP,CLY5}. \ 

Our third main result provides a complete solution of the reconstruction-of-the-measure problem for canonical invariant subspaces; we do this in Section \ref{generalROMP}. \ Briefly stated, given a finite family of Berger measures $\nu_{\bk}$ associated with a canonical invariant subspace $\mathcal{S}$, satisfying natural compatibility conditions of the type described in (\ref{compatnew}), and given $1$-variable marginal Berger measures $\sigma$ and $\tau$, the solubility of ROMP for $\W$ is fully determined by the solubility of the localized version of ROMP on the pair of subspaces $\mathcal{L}_{\mathbf{P}}$ and $\mathcal{L}_{\mathbf{q}}$, where $\mathbf{p}$ and $\mathbf{q}$ are two arbitrary lattice points in the so-called foundation set of the subspace $\mathcal{S}$. \ Once again, the natural necessary conditions for the solubility of ROMP are sufficient.   


\section{Notation and Preliminaries}

For the reader's convenience, in this section, we gather several well known
auxiliary results which are needed for the proofs of the main results in
this article. \ We recall some auxiliary facts needed for the proof of our
main results. \ In the single variable case, if $W_{\omega }$ is subnormal
with Berger measure $\sigma $, and if we let $\mathcal{L}_{j}:=\bigvee
\{e_{n}: n\geq j\}$ denote the invariant subspace obtained by
removing the first $j$ vectors in the canonical orthonormal basis of $\ell
^{2}(\mathbb{Z}_{+})$, then the Berger measure of $W_{\omega }|_{\mathcal{L}%
_{j}}$ is
\begin{equation}
d\sigma _{j}\left( s\right) =\frac{s^{j}}{\gamma _{j}}d\sigma(s) \quad (j=1,2, \ldots).   \label{measure1}
\end{equation}%

\begin{lemma}
\label{backward}(Subnormal backward extension of a $1$-variable weighted
shift) (cf. \cite[Proposition 8]{QHWS}, \cite[Proposition 1.5]{CuYo1}) \ Let $%
W_{\omega }|_{\mathcal{L}_{1}}$ be subnormal, with Berger measure $\mu _{%
\mathcal{L}_{1}}$. \ Then $W_{\omega }$ is subnormal (with Berger measure $%
\mu $) if and only if the following conditions hold:\newline
(i) $\ \frac{1}{s}\in L^{1}(\mu _{\mathcal{L}_{1}})\smallskip $\newline
(ii) $\ \omega _{0}^{2}\leq \left( \left\Vert \frac{1}{s}\right\Vert
_{L^{1}\left( \mu _{\mathcal{L}_{1}}\right) }\right) ^{-1}.$\newline
In this case, $d\mu (s)=\frac{\omega _{0}^{2}}{s}d\mu _{\mathcal{L}%
_{1}}(s)+\left( 1-\omega _{0}^{2}\left\Vert \frac{1}{s}\right\Vert
_{L^{1}\left( \mu _{\mathcal{L}_{1}}\right) }\right) d\delta _{0}(s)$, where
$\delta _{0}$ denotes the point-mass probability measure with support the
singleton set $\left\{ 0\right\} $. \ In particular, $W_{\omega }$ is never
subnormal when $\mu _{\mathcal{L}_{1}}(\{0\})>0$.
\end{lemma}

To check the subnormality of $2$-variable weighted shifts, we need to
introduce some definitions.

\begin{definition}
(\cite{CuYo1}, \cite{CuYo2}, \cite{Yo}) \ (i) \ Let $\mu $ and $\nu $ be two
positive measures on $X\equiv \mathbb{R}_{+}.$ \ We say that $\mu \leq \nu $
on $X,$ if $\mu (E)\leq \nu (E)$ for all Borel subset $E\subseteq X$;
equivalently, $\mu \leq \nu $ if and only if $\int fd\mu \leq \int fd\nu $
for all $f\in C(X)$ such that $f\geq 0$ on $X$.$\smallskip $\newline
(ii)\ \ Let $\mu $ be a probability measure on $X\times Y,$ with $Y\equiv
\mathbb{R}_{+}$, and assume that $\frac{1}{s}\in L^{1}(\mu )$ (resp $\frac{1%
}{t}\in L^{1}(\mu )$). \ The \emph{extremal measure} $\mu _{ext;s}$ (resp. $\mu _{ext;t}$) (which is also a
probability measure) on $X\times Y$ is given by
\begin{equation*}
\begin{tabular}{l}
$d\mu _{ext;s}(s,t):=\frac{1}{s\left\Vert \frac{1}{s}%
\right\Vert _{L^{1}(\mu )}}d\mu (s,t)$ $\quad$ (resp. $d\mu _{ext;t}(s,t):=\frac{1}{t\left\Vert \frac{1}{t}\right\Vert _{L^{1}(\mu )}}d\mu
(s,t) $).$\smallskip $%
\end{tabular}%
\end{equation*}
\newline
(iii) \ Given a measure $\mu $ on $X\times Y,$ the \emph{marginal measure}
$\mu ^{X}$ is given by $\mu ^{X}:=\mu \circ \pi _{X}^{-1}$, where $\pi
_{X}:X\times Y\rightarrow X$ is the canonical projection onto $X$. \ Thus $%
\mu ^{X}(E)=\mu (E\times Y)$, for every $E\subseteq X$; equivalently, $d\mu
^{X}(s)=\int_{Y}d\mu (s,t)$. \ (Observe that $\mu ^{X}$ is a probability
measure whenever $\mu $ is.)
\end{definition}

The following result is a very special case of the Reconstruction-of-the-Measure Problem.

\begin{lemma}
\label{backext} (Subnormal backward extension of a $2$-variable weighted
shift \cite{CuYo1}) \ Consider the $2$-variable weighted shift whose weight
diagram is given in Figure 1(i). \ Assume that $\mathbf{T}|_{\mathcal{L}_{(0,1)}}$
is subnormal, with associated measure $\mu_{(0,1)}$, and that $W_{0}\equiv
shift(\alpha _{00},\alpha _{10},\cdots )$ is subnormal with associated
measure $\sigma$. \ Then $\mathbf{T}$ is subnormal if and only if\newline
(i) $\frac{1}{t}\in L^{1}(\mu_{(0,1)})$;$\smallskip $\newline
(ii) $\beta _{00}^{2}\leq (\left\Vert \frac{1}{t}\right\Vert _{L^{1}(\mu_{(0,1)}
)})^{-1}$;$\smallskip $\newline
(iii) $\beta _{00}^{2}\left\Vert \frac{1}{t}\right\Vert _{L^{1}(\mu_{(0,1)})}((\mu_{(0,1)})_{ext;t})^{X}\leq \sigma$.

Moreover, if $\beta _{00}^{2}\left\Vert \frac{1}{t}\right\Vert _{L^{1}(\mu_{(0,1)})}=1$ then $((\mu_{(0,1)})_{ext;t})^{X}=\sigma$. \ In the case when $\mathbf{T}$ is
subnormal, the Berger measure $\mu $ of $\mathbf{T}$ is given by
\begin{equation}
d\mu (s,t)=\beta _{00}^{2}\left\Vert \frac{1}{t}\right\Vert _{L^{1}(\mu_{(0,1)})}d(\mu_{(0,1)})_{ext;t}(s,t)+[d\sigma(s)-\beta
_{00}^{2}\left\Vert \frac{1}{t}\right\Vert _{L^{1}(\mu_{(0,1)})}d((\mu_{(0,1)})_{ext;t})^{X}(s)]d\delta _{0}(t).\smallskip  \label{measure 1}
\end{equation}
\end{lemma}

\begin{observation}
Since the extremal measure $(\mu_{(0,1)})_{ext;t}$ involves normalization (to obtain again a probability measure), it is easy to see that the expression  $\left\Vert \frac{1}{t}\right\Vert _{L^{1}(\mu_{(0,1)})}(\mu_{(0,1)})_{ext;t}$ is equivalent to the expression $\ddfrac{\mu_{(0,1)}}{t}$; we will often use the latter expression. \ For example, Lemma \ref{backext}(iii) can be rewritten as $\beta_{00}^2 (\ddfrac{\mu_{(0,1)}}{t})^X \le \sigma$.
\end{observation}

Recall that given two positive measures $\mu $ and $\nu $ on $\mathbb{R}_{+}$%
, $\mu $ is said to be \emph{absolutely continuous} with respect to $\nu $
(in symbols, $\mu \ll \nu $) if $\nu (E)=0\Rightarrow \mu (E)=0$ for every
Borel set $E$ on $\mathbb{R}_{+}$.\smallskip

Given a $2$-variable weighted shift $\mathbf{T}=\W$ such that $T_1T_2=T_2T_1$ and $T_i$ is subnormal ($i=1,2$), and given $k_{1},k_{2}\geq 0$, we let%
\begin{equation}
W_{k_{2}}:=\mathrm{shift}(\alpha _{\left( 0,k_{2}\right) },\alpha _{\left(
1,k_{2}\right) },\cdots )\smallskip  \label{Wrow}
\end{equation}%
be the $k_{2}$-th horizontal slice of $T_{1}$ with associated Berger measure
$\xi _{\alpha _{k_{2}}}$; similarly we let
\begin{equation}
V_{k_{1}}:=\mathrm{shift}(\beta _{\left( k_{1},0\right) },\beta _{\left(
k_{1},1\right) },\cdots )\smallskip  \label{Vcol}
\end{equation}%
be the $k_{1}$-th vertical slice of $T_{2}$ with associated Berger measure $%
\eta _{\beta _{k_{1}}}$. \ (Clearly, $W_{0}$ and $V_{0}$ are the unilateral
weighted shifts associated with the $0$-th row and $0$-column in the weight
diagram for $\mathbf{T}$, resp.).\newline

\begin{lemma}
\label{DOM2} (\cite[Theorem 3.1]{CuYo2}) \ Let $\mu $ be the Berger measure
of a subnormal $2$-variable weighted shift, and let $\sigma $ (resp. $\tau$) be the Berger
measure of the associated $0$-th horizontal (resp. vertical) $1$-variable $\mathrm{shift}$ $%
(\alpha _{00},\alpha _{10},\alpha_{20},\cdots )$ (resp. $\mathrm{shift}$ $%
(\beta _{00},\beta _{01},\beta_{02},\cdots ))$. \ Then\ $\sigma =\mu ^{X}$ (resp. $\tau = \mu^Y$).
\end{lemma}

\begin{lemma}
\label{Lem3} (\cite{LLY}) \ If $\mu $ is a positive regular Borel measure
defined on $Z:=X\times Y\equiv \mathbb{R}_{+}\times \mathbb{R}_{+}$ and $%
\frac{1}{t}\in L^{1}(\mu )$, then
\begin{equation*}
\left\Vert \frac{1}{t}\right\Vert _{L^{1}(\mu )}=\left\Vert \frac{1}{t}%
\right\Vert _{L^{1}\left( \mu ^{Y}\right) }.
\end{equation*}%
\end{lemma}

\begin{remark}
For the reader's convenience, throughout the paper we adopt the convention of denoting $\mathbb{R}_+ \times \mathbb{R}_+$ by $X \times Y$, with independent variables $s \in X$ and $t \in Y$. \ Consistent with this, the marginal Berger measures are denoted by $\sigma$ (or $d \sigma(s)$) and $\tau$ (or $d \tau(t)$). \ We have found that doing so keeps a clear distinction, at the level of Berger measures and marginal measures, between the two components $T_1$ and $T_2$ of $\W$.
\end{remark} 


\section{Generalized One-step Reconstruction-of-the-Measure Problem} 

In \cite[Theorem 1.7]{CLY5} we solved the Reconstruction-of-the-Measure Problem (ROMP) for $2$-variable weighted shifts $\W$ under the assumption of subnormality for the restrictions \newline $\W |_{\mathcal{L}_{(1,0)}}$ and $\W |_{\mathcal{L}_{(0,1)}}$, together with a natural compatibility condition; that is, if $\mu_{(1,0)}$ and $\mu_{(0,1)}$ are the respective Berger measures of the restrictions, then 
$$
s \; d\mu_{(0,1)}(s,t)=\lambda t \; d\mu_{(1,0)}(s,t)
$$
for some $\lambda > 0$ (see Figure \ref{ROMPkl}(left)). \ As mentioned in \cite[Remark 3.4]{CLY5}, the proof of that result showed that, using the same techniques, one can solve a slightly more general Reconstruction-of-the-Measure Problem, as follows. \ We now state the more general result, but we leave the details to the reader. 

\begin{theorem} \label{thmgeneralized} (Generalized One-Step Extension; cf. \cite[Remark 3.4]{CLY5}) \ Consider the ROMP in Figure \ref{ROMPkl}(left), and let $\mu_{(k,0)}$ and $\mu_{(0,\ell)}$ be the Berger measures of the restrictions of $\W$ to $\mathcal{L}_{(k,0)}$ and $\mathcal{L}_{(0,\ell)}$, respectively. \ Assume that $s^k \; d\mu_{(0,\ell)}(s,t)=\lambda t^{\ell} \; d\mu_{(k,0)}(s,t)$, where $\lambda:=\frac{\gamma_{(k,0)}}{\gamma_{(0,\ell)}}>0$. \ Then $\W$ is subnormal (with Berger measure $\mu$) if and only if
\begin{itemize}
\item[(i)] \ $\frac{1}{t^{\ell}} \in L^1(\mu_{(0,\ell)})$;
\item[(ii)] \ $\frac{1}{s^{k}} \in L^1(\mu_{(k,0)})$;
\item[(iii)] \ $\gamma_{(k,0)} \left\|\frac{1}{s^k}\right\|_{L^1(\mu_{(k,0)})} = \lambda \gamma_{(0,\ell)} \left\|\frac{1}{s^k}\right\|_{L^1(\mu_{(k,0)})} \le 1$; and
\item[(iv)] \ $\gamma_{(0,\ell)} \{ \int_Y \frac{d \; \mu_{(0,\ell)}(s,t)}{t^{\ell}}+\lambda \left\|\frac{1}{s^k}\right\|_{L^1(\mu_{(k,0)})} d \delta_0(s)-\frac{\lambda}{s^k} \int_Y d\mu_{(k,0)}(s,t) \} \le d \delta_0(s)$.
\end{itemize}
In the case when $\W$ is subnormal, the Berger measure of $\W$ is given by
$$
d \mu(s,t) = \gamma_{(0,\ell)} \ddfrac{d \mu_{(0,\ell)}(s,t)}{t^{\ell}}+\left(d \sigma(s)-\gamma_{(0,\ell)} \int_Y \ddfrac{d \mu_{(0,\ell)}(s,t)}{t^{\ell}}\right) d \delta_0(t) .
$$
\end{theorem}
\bigskip


\section{Two-Step Extensions}
  
We will now pose and solve the ROMP when even less information is available.

\begin{problem} \label{twostep}
For $k,\ell>0$, consider the canonical invariant subspace $\mathcal{L}_{(k,\ell)}$ (see Figure \ref{ROMPkl}(right)). \ (When $k=\ell=1$, $\mathcal{L}_{(1,1)}$ is the \emph{core} of $\W$.) \ Assume that (i) $\W |_{\mathcal{L}_{(k,\ell)}}$ is subnormal with (two-variable) Berger measure $\nu$; (ii) $\W |_{\mathcal{L}_{(1,0)}}$ is subnormal with (one-variable) Berger measure $\sigma$; and (iii) $\W |_{\mathcal{L}_{(0,1)}}$ is subnormal with (one-variable) Berger measure $\tau$. \ Find necessary and sufficient conditions on $\nu$, $\sigma$ and $\tau$ such that $\W$ is subnormal. \ When $\W$ is indeed subnormal, find its Berger measure $\mu$ in terms of the initial data $\nu$, $\sigma$ and $\tau$.  
\end{problem}

\setlength{\unitlength}{1mm} \psset{unit=1mm} 
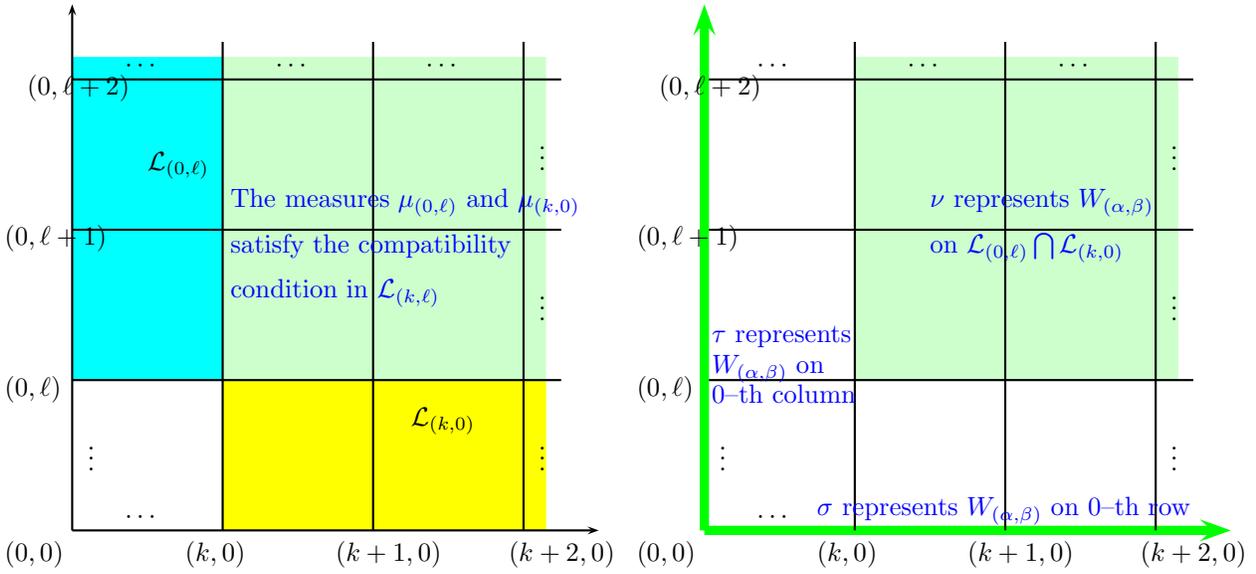
\begin{figure}[ht]
\begin{center}
\begin{picture}(105,85)

\rput(-121,25){

\pspolygon*[linecolor=lightgreen](120,40)(120,83)(163,83)(163,40)
\pspolygon*[linecolor=lightyellow](120,20)(120,40)(163,40)(163,20)
\pspolygon*[linecolor=lightblue](100,40)(120,40)(120,83)(100,83)

\psline{->}(100,20)(170,20)
\psline(100,40)(165,40)
\psline(100,60)(165,60)
\psline(100,80)(165,80)
\psline{->}(100,20)(100,90)
\psline(120,20)(120,85)
\psline(140,20)(140,85)
\psline(160,20)(160,85)

\put(91,16){\footnotesize{$(0,0)$}}
\put(115,16){\footnotesize{$(k,0)$}}
\put(135,16){\footnotesize{$(k+1,0)$}}
\put(158,16){\footnotesize{$(k+2,0)$}}

\put(107,21){\footnotesize{$\cdots$}}
\put(107,81){\footnotesize{$\cdots$}}
\put(127,81){\footnotesize{$\cdots$}}
\put(147,81){\footnotesize{$\cdots$}}

\put(91,38){\footnotesize{$(0,\ell)$}}
\put(91,58){\footnotesize{$(0,\ell+1)$}}
\put(94,78){\footnotesize{$(0,\ell+2)$}}

\put(102,28){\footnotesize{$\vdots$}}
\put(162,28){\footnotesize{$\vdots$}}
\put(162,48){\footnotesize{$\vdots$}}
\put(162,68){\footnotesize{$\vdots$}}

\put(145,34){\footnotesize{$\mathcal{L}_{(k,0)}$}}
\put(110,68){\footnotesize{$\mathcal{L}_{(0,\ell)}$}}

\put(121,63){\footnotesize{{\color{blue}The measures $\mu_{(0,\ell)}$ and $\mu_{(k,0)}$}}} 
\put(121,57){\footnotesize{{\color{blue} satisfy the compatibility}}}
\put(121,51){\footnotesize{{\color{blue} condition in $\mathcal{L}_{(k,\ell)}$}}}

}


\rput(43,25){

\pspolygon*[linecolor=lightgreen](40,40)(40,83)(83,83)(83,40)

\psline[linecolor=green,linewidth=3.5pt]{->}(20,20)(90,20)
\psline(20,40)(85,40)
\psline(20,60)(85,60)
\psline(20,80)(85,80)
\psline[linecolor=green,linewidth=3.5pt]{->}(20,20)(20,90)
\psline(40,20)(40,85)
\psline(60,20)(60,85)
\psline(80,20)(80,85)

\put(11,16){\footnotesize{$(0,0)$}}
\put(35,16){\footnotesize{$(k,0)$}}
\put(55,16){\footnotesize{$(k+1,0)$}}
\put(78,16){\footnotesize{$(k+2,0)$}}

\put(27,21){\footnotesize{$\cdots$}}
\put(27,81){\footnotesize{$\cdots$}}
\put(47,81){\footnotesize{$\cdots$}}
\put(67,81){\footnotesize{$\cdots$}}

\put(11,38){\footnotesize{$(0,\ell)$}}
\put(11,58){\footnotesize{$(0,\ell+1)$}}
\put(14,78){\footnotesize{$(0,\ell+2)$}}

\put(22,28){\footnotesize{$\vdots$}}
\put(82,28){\footnotesize{$\vdots$}}
\put(82,48){\footnotesize{$\vdots$}}
\put(82,68){\footnotesize{$\vdots$}}

\put(50,63){\footnotesize{{\color{blue}$\nu$ represents $\W$}}} 
\put(50,57){\footnotesize{{\color{blue}on $\mathcal{L}_{(0,\ell)}\bigcap \mathcal{L}_{(k,0)}$}}}
\put(35,22){\footnotesize{{\color{blue}$\sigma$ represents $\W$ on $0$--th row}}}
\put(21,45){\footnotesize{{\color{blue}$\tau$ represents}}}
\put(21,41){\footnotesize{{\color{blue}$\W$ on}}}
\put(21,37){\footnotesize{{\color{blue}$0$--th column}}}

}

\end{picture}
\end{center}
\caption{(left) The subspaces $\mathcal{L}_{(0,\ell)}$, $\mathcal{L}_{(k,0)}$ and $\mathcal{L}_{(k,\ell)}:=\mathcal L_{(0,\ell)} \bigcap \mathcal{L}_{(k,0)}$} used in Theorem \ref{thmgeneralized}; (right) initial data of two-step ROMP for the subspace $\mathcal{L}_{(k,\ell)}$ . 
\label{ROMPkl}
\end{figure}

Our strategy for solving Problem \ref{twostep} is to build an equivalent problem, whose solution can be found using Theorem \ref{thmgeneralized}. \ To this end, consider the two closed subspaces $\mathcal{L}_{(k,0)}$ and $\mathcal{L}_{(0,\ell)}$. \ We focus first on $\mathcal{L}_{(k,0)}$, and look for necessary conditions for the subnormality of $\W |_{\mathcal{L}_{(k,0)}}$. \ From Lemma \ref{backext}, it is clear that for $\nu$ to extend to the Berger measure of $\W |_{\mathcal{L}_{(k,0)}}$, we need \newline
(i) \ $\ddfrac{1}{t^{\ell}} \in L^1(\nu)$ \newline
(as a simple application of Lemma \ref{backext}, $\ell$ times, reveals). \ Similarly, we must require the conditions \newline
(ii) $\ddfrac{\gamma_{(k,\ell)}}{\gamma_{(k,0)}} \le  \left\|\ddfrac{1}{t^{\ell}}\right\|_{L^1(\nu)}^{-1}$, and \newline
(iii) $\gamma_{(k,\ell)} \displaystyle\int_Y\ddfrac{d \nu(s,t)}{t^{\ell}} \le s^k d \sigma(s)$ \newline
(again, by repeated application of Lemma \ref{backext} and the fact that the subnormal unilateral weighted shift associated with the sequence $\alpha_{(k,0)},\alpha_{(k+1,0)},\ldots$ has Berger measure $s^k\sigma$, properly normalized (cf. (\ref{measure1}))). \ 

It now follows that $\W |_{\mathcal{L}_{(k,0)}}$ is subnormal if and only if (i), (ii) and (iii) hold, and in that case, the Berger measure is reconstructed (using Lemma \ref{backext}) as 
\begin{equation} \label{mu1}
d \mu_{(k,0)}(s,t):= \ddfrac{\gamma_{(k,\ell)}}{\gamma_{(k,0)}} \ddfrac{d \nu(s,t)}{t^{\ell}}+[\ddfrac{s^kd \sigma(s)}{\gamma_{(k,0)}}-\ddfrac{\gamma_{(k,\ell)}}{\gamma_{(k,0)}} \int_Y \ddfrac{d \nu(s,t)}{t^{\ell}}] \cdot d \delta_0(t)
\end{equation}
(observe that the measure inside the square brackets is a one-variable measure in $s$).

Having dealt with the subspace $\mathcal{L}_{(k,0)}$, it is now straightforward to list the necessary and sufficient conditions for $\W$ to be subnormal on the subspace $\mathcal{L}_{(0,\ell)}$; these are \newline
(i) \ $\ddfrac{1}{s^{k}} \in L^1(\nu)$ \newline
(ii) $\ddfrac{\gamma_{(k,\ell)}}{\gamma_{(0,\ell)}} \le  \left\|\ddfrac{1}{s^{k}}\right\|_{L^1(\nu)}^{-1}$, and \newline
(iii) $\gamma_{(k,\ell)} \displaystyle \int_X \ddfrac{d \nu(s,t)}{s^{k}} \le t^{\ell} d \tau(t)$. \newline
When these conditions hold, $\W |_{\mathcal{L}_{(0,\ell)}}$ is subnormal, with Berger measure given by  
\begin{equation} \label{mu2}
d \mu_{(0,\ell)}(s,t):= \ddfrac{\gamma_{(k,\ell)}}{\gamma_{(0,\ell)}} \ddfrac{d \nu(s,t)}{s^{k}}+ d \delta_0(s) \cdot [\ddfrac{t^{\ell}d \tau(t)}{\gamma_{(0,\ell)}}-\ddfrac{\gamma_{(k,\ell)}}{\gamma_{(0,\ell)}} \int_X \ddfrac{d \nu(s,t)}{s^k}] 
\end{equation}
(with the measure inside the square brackets a one-variable measure in $t$).

We are now ready to state and prove a solution to Problem \ref{twostep}.

\begin{theorem} \label{bigthm}
Let $\W\equiv \mathbf{(}T_{1},T_{2})$ be a commuting $2$-variable weighted shift, and let $k,\ell>0$. \ Assume that the unilateral weighted shifts corresponding to the $0$--th row $R_{0}$ and the $0$--th column $C_{0}$ are subnormal, with Berger measures $\sigma$ and $\tau$, respectively. \ Assume also that $\W |_{\mathcal{L}_{(0,\ell)}} \bigcap \mathcal{L}_{(k,0)}$ is subnormal with Berger measure $\nu$. \ Then $\W$ is subnormal (with Berger measure $\mu$) if and only if 
\begin{enumerate}
\item[(NC1)] \ $\ddfrac{1}{s^kt^{\ell}}\in L^{1}(\nu )$ \medskip
\item[(NC2)] \ $\ddfrac{\gamma_{(k,\ell)}}{\gamma_{(0,\ell)}}\int_{X} \ddfrac{d \nu(s,t)}{s^k}\leq d \tau_{\ell}(t):=\ddfrac{t^{\ell} d \tau(t)}{\gamma_{(0,\ell)}}$ \medskip
\item[(NC3)] \ $\ddfrac{\gamma_{(k,\ell)}}{\gamma_{(k,0)}}\int_{Y} \ddfrac{d \nu(s,t)}{t^{\ell}}\leq d \sigma_{k}(s):=\ddfrac{s^{k}d \sigma(s)}{\gamma_{(k,0)}}$ \medskip
\item[(NC4)] \ $\int_{X\times Y} \ddfrac{d \nu(s,t)}{s^kt^{\ell}} = \ddfrac{1}{\gamma_{(k,\ell)}}$ .
\end{enumerate}
When $\W$ is subnormal, the Berger measure $\mu$ is given by
\begin{equation} \label{formula1}
d \mu(s,t)= \gamma_{(k,\ell)} \ddfrac{d \nu(s,t)}{s^kt^{\ell}} + \left(d \sigma(s) - \gamma_{(k,\ell)} \int_Y \ddfrac{d \nu(s,t)}{s^kt^{\ell}} \right) d \delta_0(t) + d \delta_0(s) \left( d \tau(t)-\gamma_{(k,\ell)} \int_X \ddfrac{d \nu(s,t)}{s^kt^{\ell}} \right) ,
\end{equation}
or in the equivalent (succinct) form
\begin{equation} \label{formula2}
\mu = \gamma_{(k,\ell)} \ddfrac{\nu}{s^kt^{\ell}} + \left(\sigma - \gamma_{(k,\ell)} \int_Y \ddfrac{\nu}{s^kt^{\ell}} \right) \times \delta_0 + \delta_0 \times \left(\tau-\gamma_{(k,\ell)} \int_X \ddfrac{\nu}{s^kt^{\ell}} \right) .
\end{equation}
\end{theorem}

\begin{remark}
Observe that each of the one-variable measures inside the parentheses in (\ref{formula2}) is positive, by (NC3) and (NC2), respectively. \ As a result, the Berger measure $\mu$ is the sum of three positive measures, involving the initial data ($\nu$, $\sigma$ and $\tau$) together with the marginal measures $\nu^X$ and $\nu^Y$.
\end{remark}

\begin{proof}[Proof of Theorem \ref{bigthm}]
We wish to apply the natural generalization of Lemma \ref{backext} when the subspace $\mathcal{M}$ is replaced by $\mathcal{L}_{(k,0})$. \ It is straightforward to list the three required conditions:
\begin{enumerate}
\item[(i)] \ $\ddfrac{1}{t^{\ell}} \in L^1(\nu)$ ; \medskip
\item[(ii)] \ $\ddfrac{\gamma_{(k,\ell)}}{\gamma_{(k,0)}} \le (\left\|\ddfrac{1}{t^{\ell}}\right\|_{L^1(\nu)})^{-1}$ ; \medskip
\item[(iii)] \ $\ddfrac{\gamma_{(k,\ell)}}{\gamma_{(k,0)}} \int_Y \ddfrac{d \nu(s,t)}{t^{\ell}} \le d \sigma_{k}(s)$ .
\end{enumerate}
Observe first that (NC1) implies that both $\frac{1}{s^k}$ and $\frac{1}{t^{\ell}}$ belong to $L^1(\nu)$; for instance, $\frac{1}{t^{\ell}}=s^k \cdot \frac{1}{s^kt^{\ell}}$, and $s^k \in L^{\infty}(\nu)$. \ Thus, the first hypothesis in Lemma \ref{backext} is satisfied. \ To get (ii), we integrate (NC3) with respect to $s$; that is, 
$$
\ddfrac{\gamma_{(k,\ell)}}{\gamma_{(k,0)}}\int_X \int_{Y} \ddfrac{d \nu(s,t)}{t^{\ell}}\leq \int_X d \sigma_{k}(t) = \int_X \ddfrac{s^{k}d \sigma(s)}{\gamma_{(k,0)}} .
$$
It follows that 
$$
\ddfrac{\gamma_{(k,\ell)}}{\gamma_{(k,0)}} \left\|\ddfrac{1}{t^{\ell}}\right\|_{L^1(\nu)} \le \ddfrac{\gamma_{(k,0)}}{\gamma_{(k,0)}}=1,
$$
as desired. \ We now observe that (iii) and (NC3) are identical. \ Having verified (i), (ii) and (iii), we conclude that $\nu$ admits a back-step extension $\mu_{(k,0)}$ to the subspace $\mathcal{L}_{(k,0)}$.

In a completely similar way, we use (NC1) and (NC2) to obtain a back-step extension of $\nu$, denoted $\mu_{(0,\ell)}$, to the subspace $\mathcal{L}_{(0,\ell)}$. 

We now refer to the hypotheses for Theorem \ref{thmgeneralized} for the case of two subspaces $\mathcal{L}_{(k,0)}$ and $\mathcal{L}_{(0,\ell)}$, and Berger measures $\mu_{(k,0)}$ and $\mu_{(0,\ell)}$, respectively. \ First, we need to check the compatibility condition $s^k d \mu_{(0,\ell)}(s,t) = \lambda t^{\ell} d \mu_{(k,0)}(s,t)$, with $\lambda=\frac{\gamma_{(k,0)}}{\gamma_{(0,\ell)}}$. \ By (\ref{mu2}), we have 
$$
s^k d \mu_{(0,\ell)}(s,t) = \frac{\gamma_{(k,\ell)}}{\gamma_{(0,\ell)}} d \nu(s,t)+ s^k d \delta_0(s) \cdot [d\; \tau(t)-\frac{\gamma_{(k,\ell)}}{\gamma_{(0,\ell)}} \int_X d \; \nu(s,t)] = \frac{\gamma_{(k,\ell)}}{\gamma_{(0,\ell)}} d \nu(s,t) . 
$$
On the other hand, by (\ref{mu1}),
$$
\lambda t^{\ell} d \mu_{(k,0)}(s,t) = \frac{\gamma_{(k,\ell)}}{\gamma_{(0,\ell)}} d \nu(s,t)+ \frac{\gamma_{(k,0)}}{\gamma_{(0,\ell)}}[d \sigma(s)-\frac{\gamma_{(k,\ell)}}{\gamma_{(k,0)}} \int_X d \nu(s,t)] \cdot t^{\ell} d \delta_0(t)  = \frac{\gamma_{(k,\ell)}}{\gamma_{(0,\ell)}} d \nu(s,t) . 
$$
It is now clear that the compatibility condition holds. \ We now calculate the $L^1$--norms of $\frac{1}{t^{\ell}}$ and $\frac{1}{s^k}$, using (\ref{mu2}) and (\ref{mu1}), respectively; in each case, we use (NC4) in the last step.

\begin{eqnarray*}
\int_{X\times Y}\frac{1}{t^{\ell}}d \mu _{(0,\ell)}(s,t) &=& \frac{\gamma_{(k,\ell)}}{\gamma_{(0,\ell)}}\int_{X\times Y}\frac{%
d \nu(s,t) }{s^kt^{\ell}}+\int_{X}d \delta _{0}(s)[\int_{Y}\frac{d \tau(t) }{\gamma_{(0,\ell)}}-\frac{\gamma_{(k,\ell)}}{\gamma_{(0,\ell)}}\int_{Y}\int_{X}\frac{d\nu(s,t) }{s^kt^{\ell}}] \\
&=&\frac{1}{\gamma_{(0,\ell)}}<\infty
\end{eqnarray*}%
and%
\begin{eqnarray*}
\int_{X\times Y}\frac{1}{s^k} d \mu_{(k,0)}(s,t) &=&\frac{\gamma_{(k,\ell)}}{\gamma_{(k,0)}} \int_{X\times Y}\frac{%
d \nu(s,t) }{s^kt^{\ell}}+[\int_{X}\frac{d \sigma(s)}{\gamma_{(k,0)}}-\frac{\gamma_{(k,\ell)}}{\gamma_{(k,0)}}\int_{X}\int_{Y}\frac{d \nu(s,t)}{s^kt^{\ell}}]\int_{Y}d \delta_{0}(t) \\
&=&\frac{1}{\gamma_{(k,0)}}<\infty .
\end{eqnarray*}
Moreover,
\begin{equation*}
\lambda \gamma_{(0,\ell)}\left\Vert \frac{1}{s^k}\right\Vert
_{L^{1}(\mu_{(k,0)})}=\gamma_{(k,0)}\cdot \frac{1}{\gamma_{(k,0)}}=1.
\end{equation*}
Finally, we need to check condition (iv) in Theorem \ref{thmgeneralized}:
$$
\gamma_{(0,\ell)} \{ \int_Y \frac{d \mu_{(0,\ell)}(s,t)}{t^{\ell}}+\lambda \left\|\frac{1}{s^k}\right\|_{L^1(\mu_{(k,0)})} d \delta_0(s)-\frac{\lambda}{s^k} \int_Y d \mu_{(k,0)}(s,t) \} \le d \delta_0(s).
$$
In preparation for this, we first calculate the two integrals on the left-hand side of (iv). \ Observe that 
\begin{equation*}
\int_{Y}\ddfrac{d \mu_{(0,\ell)}}{t^{\ell}}=\ddfrac{\gamma_{(k,\ell)}}{\gamma_{(0,\ell)}}\int_{Y}\ddfrac{d \nu(s,t) }{s^kt^{\ell}}+d \delta
_{0}(s) \lbrack \int_Y \frac{d \tau(t)}{\gamma_{(0,\ell)}}-\ddfrac{\gamma_{(k,\ell)}}{\gamma_{(0,\ell)}}\int_Y \int_{X}\frac{d \nu(s,t)}{s^kt^{\ell}}\rbrack=\ddfrac{\gamma_{(k,\ell)}}{\gamma_{(0,\ell)}}\int_{Y}\ddfrac{d\nu(s,t) }{s^kt^{\ell}}
\end{equation*}%
(the last step uses both the fact that $\tau$ is a probability measure and condition (NC4), and as a result, the quantity in square brackets is $0$). \ Similarly,
\begin{equation*}
\int_{Y}d \mu_{(k,0)}(s,t)=\ddfrac{\gamma_{(k,\ell)}}{\gamma_{(0,k)}}\int_{Y}\ddfrac{d \nu(s,t)}{t^{\ell}}+(\ddfrac{s^k d \sigma(s)}{\gamma_{(k,0)}}-\ddfrac{\gamma_{(k,\ell)}}{\gamma_{(0,k)}}\int_{Y}\ddfrac{d \nu(s,t)}{t^{\ell}})\int_{Y} d \delta_{0}(t)=\ddfrac{s^k d \sigma(s)}{\gamma_{(k,0)}}.
\end{equation*}%
We are now ready to establish (iv).
\begin{eqnarray*}
\gamma_{(0,\ell)} \{ \int_Y \ddfrac{d \mu_{(0,\ell)}(s,t)}{t^{\ell}} & + & \lambda \left\|\ddfrac{1}{s^k}\right\|_{L^1(\mu_{(k,0)})} d \delta_0(s)-\ddfrac{\lambda}{s^k} \int_Y d\mu_{(k,0)}(s,t) \} \\
& = & \gamma_{(0,\ell)} \{\ddfrac{\gamma_{(k,\ell)}}{\gamma_{(0,\ell)}}\int_{Y}\ddfrac{d\nu(s,t) }{s^kt^{\ell}}+\lambda \ddfrac{1}{\gamma_{(k,0)}} d \delta_0(s) - \ddfrac{\lambda}{s^k} \ddfrac{s^k d \sigma(s)}{\gamma_{(k,0)}} \} \\
& = & \gamma_{(k,\ell)} \int_Y \ddfrac{d \nu(s,t)}{s^kt^{\ell}}+d \delta_0(s) - d \sigma(s) \; \; \; \; (\textrm{recall that } \lambda=\ddfrac{\gamma_{(k,0)}}{\gamma_{(0,\ell)}} )\\
& \le & d \sigma(s)+d \delta_0(s) - d \sigma(s)=d \delta_0(s),
\end{eqnarray*}%
where we have used (NC3) in the penultimate step. \ This proves Theorem \ref{thmgeneralized}(iv). \ By Theorem \ref{thmgeneralized}, we now know that $\W$ is subnormal. \ To complete the proof, we need to explicitly compute its Berger measure.

By Theorem \ref{thmgeneralized}, using the subspace $\mathcal{L}_{(0,\ell)}$, the measure $\mu_{(0,\ell)}$ and the moment $\gamma_{(0,\ell)}$, we can reconstruct the Berger measure of $\W$ as
$$
\mu = \gamma_{(0,\ell)} \ddfrac{\mu_{(0,\ell)}}{t^{\ell}}+\left(\sigma-\gamma_{(0,\ell)} \int_Y \ddfrac{\mu_{(0,\ell)}}{t^{\ell}}\right) \times \delta_0 .
$$
Now observe that 
$$
\gamma_{(0,\ell)}\int_Y \ddfrac{d \mu_{(0,\ell)}(s,t)}{t^{\ell}} = \gamma_{(k,\ell)} \int_Y \ddfrac{d \nu(s,t)}{s^kt^{\ell}} + d \delta_0(s) \left(1-\gamma_{(k,\ell)} \int_Y \int_X \ddfrac{d \nu(s,t)}{s^kt^{\ell}}\right)=\gamma_{(k,\ell)} \int_Y \ddfrac{d \nu(s,t)}{s^kt^{\ell}}
$$
(using (NC4) to obtain the last equality). \ Then
$$
\mu=\gamma_{(0,\ell)} \ddfrac{\mu_{(0,\ell)}}{t^{\ell}}+\left(\sigma-\gamma_{(k,\ell)} \int_Y \ddfrac{\nu}{s^kt^{\ell}}\right) \times \delta_0 . 
$$
We now use (\ref{mu2}) to obtain
$$
\mu=\gamma_{(k,\ell)} \ddfrac{d \nu}{s^{k}t^{\ell}}+ \delta_0 \times \left(\tau-\gamma_{(k,\ell)} \int_X \ddfrac{d \nu}{s^k}\right)+\left(\sigma-\gamma_{(k,\ell)} \int_Y \ddfrac{\nu}{s^kt^{\ell}}{t^{\ell}}\right) \times \delta_0 , 
$$
as desired.
\end{proof}


\section{ROMP for Canonical Invariant Subspaces} \label{generalROMP}

Let $P \subseteq \mathbb{Z}_+^2$ be such that $P+\mathbb{Z}_+^2=P$, and let $\mathcal{L}_P$ be the closed subspace of $\ell^2(\mathbb{Z}_+^2)$ generated by the orthonormal basis vectors $e_{\mathbf{k}}$, where $\mathbf{k} \in P$. \ The subspace $\mathcal{L}_P$ is invariant under $T_1$ and $T_2$, for any $2$-variable weighted shift $\W \equiv(T_1,T_2)$. 

\begin{definition} \label{canonical}
A set $P$ such that $P+\Z=P$ is said to be \emph{full}. \ A closed subspace $\mathcal{S}$ of $\ell^2(\Z)$ invariant under $T_1$ and $T_2$ is said to be \emph{canonical} if $\mathcal{S}=\mathcal{L}_P$ for some full set $P$.  
\end{definition}

\begin{remark}
Typical examples of full and non-full sets are given in Figure \ref{full}. \ Observe that for any nonempty set $P \subseteq \Z$, the set $P+\Z$ is full. \ Observe also that any full set $P$ is the union of the (full) sets $\mathbf{k}+\Z$, where $\mathbf{k}$ is a point in $P$; in symbols, $P=\bigcup_{\mathbf{k} \in P} (\mathbf{k}+\Z)$. \ Thus, $\mathcal{L}_P = \bigvee_{\mathbf{k} \in P} \mathcal{L}_{\{\mathbf{k}\}}$ (cf. the left diagram in Figure \ref{full}). \ As a consequence, if $P$ and $Q$ are full sets, and if $\mathcal{L}_P=\mathcal{L}_Q$, then $P=Q$. \ (With slight abuse of notation, we will often write $\mathcal{L}_{\mathbf{k}}$ instead of $\mathcal{L}_{\{\mathbf{k}\}}$).
\end{remark}



\begin{definition}
Let $P$ be a full set in $\Z$. \ Amongst all subsets $Q$ of $P$ satisfying $Q+\Z=P$, there is a smallest one; i.e., one that is contained in any other subset $R$ of $P$ such that $R+\Z=P$. \ We will call this minimal set the \emph{foundation} of $P$, and denote it by $\fnd(P)$. \ \ We now define $\stc(P)$ to be the unique nonincreasing path contained in $P$ and including all points in $\fnd(P)$. 
\end{definition}

\begin{remark} (i) The points of $\fnd(P)$ are characterized by the following property: $\bk \in \fnd(P)$ if and only $\bk \in P$, $\mathbf{k}-\varepsilon_1 \notin P$ and $\mathbf{k}-\varepsilon_2 \notin P$. 

(ii) \ For $P$ a full set, observe that $\fnd(P)$ is always a finite set, and that $\fnd(P) \subseteq \stc(P) \subseteq P$ (by definition). \ Also, $\fnd(P)+\Z=\stc(P)+\Z=P$. \newline 
(iii) It is easy to see that $\stc(P)$ can be represented by a descending staircase; for a visual depiction, see the green staircase in the left diagram of Figure \ref{full}. \newline
(iv) \ To list the points in $\stc(P)$, we will follow the descending staircase from left to right; that is, $\mathbf{p}\equiv(p_1,p_2)$ will come before $\mathbf{q}\equiv(q_1,q_2)$ if and only if $p_1 \le q_1$ and $p_2 \ge q_2$. \newline
(v) If $P$ is full, one can recover $P$ from $\stc(P)$ by selecting all points in $\Z$ which are located above and to the right of the descending staircase representing $\stc(P)$.
\end{remark}

\setlength{\unitlength}{1mm} \psset{unit=1mm}
\begin{figure}[th]
\begin{center}
\begin{picture}(150,80)

\rput(0,29){

\psline{->}(5,20)(70,20)
\psline(5,35)(68,35)
\psline(5,50)(68,50)
\psline(5,65)(68,65)
\psline(5,80)(68,80)

\psline{->}(5,20)(5,87)
\psline(20,20)(20,84)
\psline(35,20)(35,84)
\psline(50,20)(50,84)
\psline(65,20)(65,84)

\put(1,16){\footnotesize{$(0,0)$}}
\put(17,16){\footnotesize{$(1,0)$}}
\put(-4,35){\footnotesize{$(0,1)$}}
\put(34,16){\footnotesize{$\cdots$}}
\put(46.5,16){\footnotesize{$(k,0)$}}
\put(-4.4,65){\footnotesize{$(0,k)$}}
\put(64,16){\footnotesize{$\cdots$}}

\psframe[fillstyle=solid,fillcolor=blue](3,63)(7,67)
\psframe[fillstyle=solid,fillcolor=blue](18,48)(22,52)
\psframe[fillstyle=solid,fillcolor=blue](48,33)(52,37)

\psframe[fillstyle=solid,fillcolor=red](4,19)(6,21)
\psframe[fillstyle=solid,fillcolor=red](19,19)(21,21)
\psframe[fillstyle=solid,fillcolor=red](34,19)(36,21)
\psframe[fillstyle=solid,fillcolor=red](49,19)(51,21)
\psframe[fillstyle=solid,fillcolor=red](64,19)(66,21)

\psframe[fillstyle=solid,fillcolor=red](4,34)(6,36)
\psframe[fillstyle=solid,fillcolor=red](19,34)(21,36)
\psframe[fillstyle=solid,fillcolor=red](34,34)(36,36)
\pscircle[fillstyle=solid,fillcolor=lightblue](50,35){1.3}
\pscircle[fillstyle=solid,fillcolor=lightblue](65,35){1.3}

\psframe[fillstyle=solid,fillcolor=red](4,49)(6,51)
\pscircle[fillstyle=solid,fillcolor=lightblue](20,50){1.3}
\pscircle[fillstyle=solid,fillcolor=lightblue](35,50){1.3}
\pscircle[fillstyle=solid,fillcolor=lightblue](50,50){1.3}
\pscircle[fillstyle=solid,fillcolor=lightblue](65,50){1.3}

\pscircle[fillstyle=solid,fillcolor=lightblue](5,65){1.3}
\pscircle[fillstyle=solid,fillcolor=lightblue](20,65){1.3}
\pscircle[fillstyle=solid,fillcolor=lightblue](35,65){1.3}
\pscircle[fillstyle=solid,fillcolor=lightblue](50,65){1.3}
\pscircle[fillstyle=solid,fillcolor=lightblue](65,65){1.3}

\pscircle[fillstyle=solid,fillcolor=lightblue](5,80){1.3}
\pscircle[fillstyle=solid,fillcolor=lightblue](20,80){1.3}
\pscircle[fillstyle=solid,fillcolor=lightblue](35,80){1.3}
\pscircle[fillstyle=solid,fillcolor=lightblue](50,80){1.3}
\pscircle[fillstyle=solid,fillcolor=lightblue](65,80){1.3}

\psline[linecolor=darkgreen, linewidth=2pt](6,65)(19,65)
\psline[linecolor=darkgreen, linewidth=2pt](20,64)(20,51)
\psline[linecolor=darkgreen, linewidth=2pt](21,50)(34,50)
\psline[linecolor=darkgreen, linewidth=2pt](36,50)(49,50)
\psline[linecolor=darkgreen, linewidth=2pt](50,49)(50,36)
\psline[linecolor=darkgreen, linewidth=2pt](51,35)(64,35)

\put(35.5,55){\footnotesize{$\vdots$}}
\put(41,50.5){\footnotesize{$\cdots$}}
\put(55.5,36){\footnotesize{$\cdots$}}

\psline{->}(30,14)(45,14)
\put(35,10){$\rm{T}_1$}
\psline{->}(-0.5,45)(-0.5,60)
\put(-6,52){$\rm{T}_2$}
\put(1,48){$\vdots$}
\put(1,78){$\vdots$}
\put(76,48){$\vdots$}
\put(76,78){$\vdots$}


\psline{->}(105,14)(120,14)
\put(110,10){$\rm{T}_1$}

\psline{->}(80,20)(145,20)
\psline(80,35)(143,35)
\psline(80,50)(143,50)
\psline(80,65)(143,65)
\psline(80,80)(143,80)

\psline{->}(80,20)(80,87)
\psline(95,20)(95,84)
\psline(110,20)(110,84)
\psline(125,20)(125,84)
\psline(140,20)(140,84)

\put(76,16){\footnotesize{$(0,0)$}}
\put(92,16){\footnotesize{$(1,0)$}}
\put(109,16){\footnotesize{$\cdots$}}
\put(121.5,16){\footnotesize{$(k,0)$}}
\put(139,16){\footnotesize{$\cdots$}}

\put(110.5,55){\footnotesize{$\vdots$}}
\put(116,50.5){\footnotesize{$\cdots$}}
\put(125.5,40){\footnotesize{$\vdots$}}

\pscircle[fillstyle=solid,fillcolor=blue](80,20){1.3}
\pscircle[fillstyle=solid,fillcolor=blue](95,20){1.3}
\pscircle[fillstyle=solid,fillcolor=blue](110,20){1.3}
\psframe[fillstyle=solid,fillcolor=red](124,19)(126,21)

\pscircle[fillstyle=solid,fillcolor=blue](80,35){1.3}
\psframe[fillstyle=solid,fillcolor=red](94,34)(96,36)
\pscircle[fillstyle=solid,fillcolor=blue](110,35){1.3}
\pscircle[fillstyle=solid,fillcolor=blue](125,35){1.3}

\psframe[fillstyle=solid,fillcolor=red](79,49)(81,51)
\pscircle[fillstyle=solid,fillcolor=blue](95,50){1.3}
\psframe[fillstyle=solid,fillcolor=red](109,49)(111,51)
\psframe[fillstyle=solid,fillcolor=red](124,49)(126,51)

\pscircle[fillstyle=solid,fillcolor=blue](80,65){1.3}
\pscircle[fillstyle=solid,fillcolor=blue](95,65){1.3}
\psframe[fillstyle=solid,fillcolor=red](109,64)(111,66)
\pscircle[fillstyle=solid,fillcolor=blue](125,65){1.3}

}

\end{picture}
\end{center}
\caption{(left) Diagram of a full set {\color{lightblue}$P \;$} ({\color{lightblue}$P$}$+\Z=${\color{lightblue}$P$}), its foundation {\color{blue}$\fnd(P)$}, and its descending staircase {\color{darkgreen}$\stc(P)$}; (right) diagram of a non-full set {\color{blue} $Q \;$} ({\color{blue}$Q$}$+\Z \ne$ {\color{blue}$Q$}). \quad \; {\footnotesize{({\it Color codes}: a {\color{lightblue}light-blue disk} denotes a point in $P$, a {\color{red} red square} denotes a point not in $P$, and a {\color{blue}blue square} denotes a point in {\color{blue}$\fnd(P)$}.)}}}
\label{full}
\end{figure}
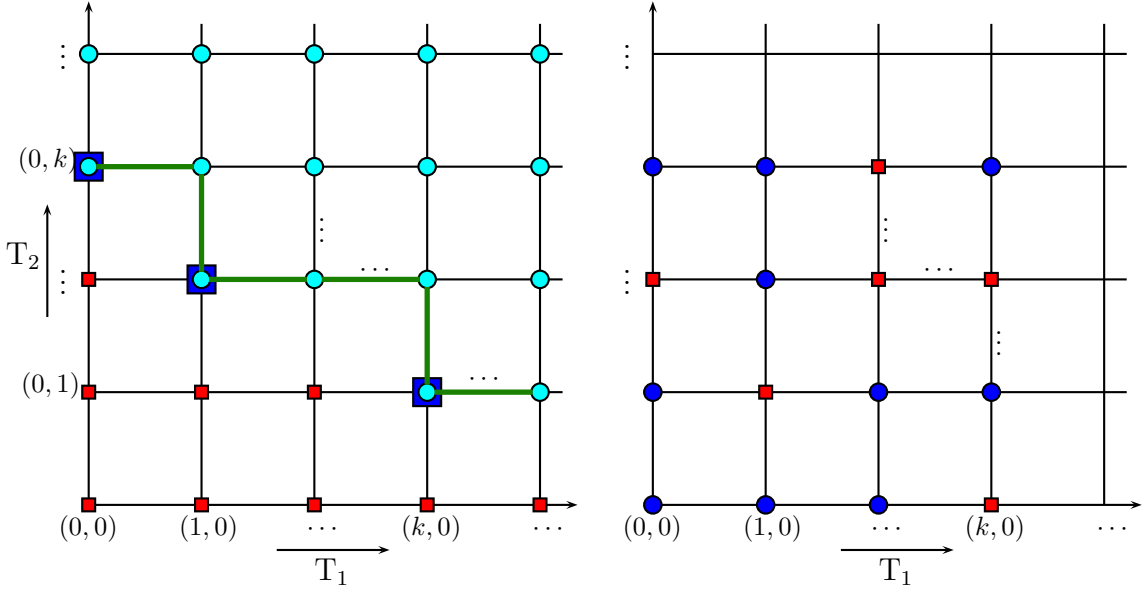


Consider now a $2$-variable weighted shift $\W$ and a canonical invariant 
subspace $\mathcal{S}$, which is necessarily of the form $\mathcal{L}_P$ for 
some $P$ full. \ Let $\fnd(P)$ be the foundation $P$; it follows that $
\mathcal{S} = \bigvee_{\mathbf{k} \in \fnd(P)} \mathcal{L}_{\mathbf{k}}$. 
\ Assume now $\W |_{\mathcal{L}_{\bk}}$ is subnormal. \ It 
immediately follows that, for each $\bk \in \fnd(P)$, the restriction of $
\W$ to the subspace $\mathcal{L}_{\mathbf{k}}$ must be subnormal. \ Now, it 
is clear that $\W |_{\lk}$ is a $2$-variable weighted shift, so it 
has a Berger measure, which we will denote by $\nu_{\bk}$. \ Thus, the 
subnormality of $\W |_{\mathcal{L}_P}$ translates into the 
existence of a finite family of Berger measures $\nu_{\bk}$, one for each point $\bk \in \fnd(P)$. \ 

\medskip There is a compatibility condition, however. \ If we consider two points $\mathbf{p}, \mathbf{q} \in \fnd(P)$, we know that the intersection of the canonical invariant subspaces $\mathcal{L}_{\mathbf{p}}$ and $\mathcal{L}_{\mathbf{q}}$ is \newline
$\mathcal{L}_{(\max\{p_1,q_1\},\max\{p_2,q_2\})}$. \ Without loss of generality, assume that $p_1 \le q_1$ and $q_2 \le p_2$. \ Then $\mathcal{R}:=\mathcal{L}_{\mathbf{p}} \bigcap \mathcal{L}_{\mathbf{q}} = \mathcal{L}_{(q_1,p_2)}$. \ By (\ref{measure1}), the Berger measure of $\W |_{\mathcal{R}}$ is given by two expressions, namely 
$$
\ddfrac{\gamma_{(p_1,p_2)}}{\gamma_{(q_1,p_2)}}s^{q_1-p_1} \nu_{(p_1,p_2)} \quad \textrm{and } \quad \ddfrac{\gamma_{(q_1,q_2)}}{\gamma_{(q_1,p_2)}}t^{p_2-q_2} \nu_{(q_1,q_2)}.
$$
\medskip
Therefore, as a necessary condition for the solubility of ROMP we must require 
$$
\ddfrac{\gamma_{(p_1,p_2)}}{\gamma_{(q_1,p_2)}}s^{q_1-p_1} \nu_{(p_1,p_2)} = \ddfrac{\gamma_{(q_1,q_2)}}{\gamma_{(q_1,p_2)}}t^{p_2-q_2} \nu_{(q_1,q_2)} ,
$$
that is, 
\begin{equation} \label{compat}
\gamma_{\mathbf{p}} s^{q_1-p_1}\nu_{\mathbf{p}} = \gamma_{\mathbf{q}} t^{p_2-q_2}\nu_{\mathbf{q}} .
\end{equation}
\medskip
We will denote by $\nu_{\mathbf{pq}}$ the Berger measure of $\W$ restricted to $\mathcal{L}_{\mathbf{p}} \bigcap \mathcal{L}_{\mathbf{q}}$. \ Similarly, we will let $\sigma_{p_1}$ and $\tau_{q_2}$ denote the restrictions of $\sigma$ to $\mathcal{L}_{p_1} (\subseteq \ell^2(\mathbb{Z}_+))$ and $\tau$ to $\mathcal{L}_{q_2} (\subseteq \ell^2(\mathbb{Z}_+))$, respectively. \ With the above mentioned compatibility condition in mind, we are ready to state and prove the solution of ROMP for arbitrary canonical invariant subspaces. \   

\begin{theorem} \label{canonicalthm}
Let $\W$ be a $2$-variable weighted shift, let $\mathcal{L}_P$ be a canonical invariant subspace, and assume that $\W |_{\mathcal{L}_P}$ is subnormal. \ Let $\fnd(P)$ be the foundation of $P$ (with points listed in descending-staircase order), and let $\{\nu_{\bk}\}_{\bk \in \fnd(P)}$ be the finite family of Berger measures for the restrictions of $\W$ to the subspaces $\mathcal{L}_{\bk} \; (\bk \in \fnd(P))$. \ Suppose that the compatibility condition (\ref{compat}) holds for every $\mathbf{p},\mathbf{q} \in \fnd(P)$. \ In addition, let $\sigma$ and $\tau$ be the Berger measures of the $0$--th row and $0$--th column in the weigh diagram of $\W$. \ The following statements are equivalent.
\begin{enumerate}
\item[(i)] \ $\W$ is subnormal.
\item[(ii)] \ For every $\mathbf{p}, \mathbf{q} \in \fnd(P)$, the ROMP with initial data $\nu_{\mathbf{pq}}$, $\sigma_{p_1}$ and $\tau_{q_2}$ is soluble.
\item[(iii)] \ For every consecutive pair of points $\mathbf{p}, \mathbf{q} \in \fnd(P)$ (in the descending-staircase order), the ROMP with initial data $\nu_{\mathbf{pq}}$, $\sigma_{p_1}$ and $\tau_{q_2}$ is soluble.
\item[(iv)] \ For every $\mathbf{p}, \mathbf{q} \in \fnd(P)$, with $p_1 \ne q_1$ and $p_2 \ne q_2$, the ROMP with initial data $\nu_{\mathbf{pq}}$, $\sigma_{p_1}$ and $\tau_{q_2}$ is soluble.
\end{enumerate}
\end{theorem}

\begin{proof}
Looking at the staircase diagram in Figure \ref{full}(left), we can easily see that four basic and distinct descending-staircase types for $\fnd(P)$ arise. \ We exhibit these types in Figures \ref{Figure 4} and \ref{Figure 5}. \ We now analyze ROMP for each descending-staircase type.

\setlength{\unitlength}{1mm} \psset{unit=1mm}
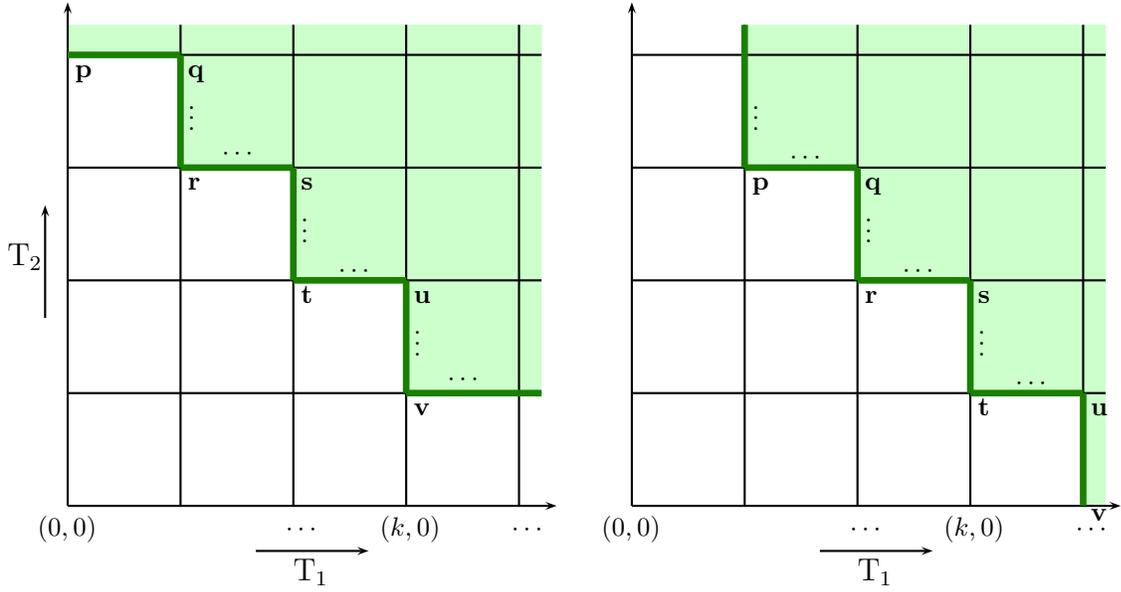
\begin{figure}[th]
\begin{center}
\begin{picture}(150,80)

\rput(0,26){

\pspolygon*[linecolor=lightgreen](5,84)(5,80)(68,80)(68,84)(5,84)
\pspolygon*[linecolor=lightgreen](20,84)(20,65)(68,65)(68,84)(20,84)
\pspolygon*[linecolor=lightgreen](35,84)(35,50)(68,50)(68,84)(35,84)
\pspolygon*[linecolor=lightgreen](50,84)(50,35)(68,35)(68,84)(50,84)

\psline{->}(5,20)(70,20)
\psline(5,35)(68,35)
\psline(5,50)(68,50)
\psline(5,65)(68,65)
\psline(5,80)(68,80)

\psline{->}(5,20)(5,87)
\psline(20,20)(20,84)
\psline(35,20)(35,84)
\psline(50,20)(50,84)
\psline(65,20)(65,84)

\psline[linecolor=darkgreen, linewidth=2.5pt](5,80)(20,80)
\psline[linecolor=darkgreen, linewidth=2.5pt](20,80)(20,65)
\psline[linecolor=darkgreen, linewidth=2.5pt](20,65)(35,65)
\psline[linecolor=darkgreen, linewidth=2.5pt](35,65)(35,50)
\psline[linecolor=darkgreen, linewidth=2.5pt](35,50)(50,50)
\psline[linecolor=darkgreen, linewidth=2.5pt](50,50)(50,35)
\psline[linecolor=darkgreen, linewidth=2.5pt](50,35)(68,35)

\put(1,16){\footnotesize{$(0,0)$}}
\put(34,16){\footnotesize{$\cdots$}}
\put(46.5,16){\footnotesize{$(k,0)$}}
\put(64,16){\footnotesize{$\cdots$}}

\put(6,77){\footnotesize{$\mathbf{p}$}}
\put(21,77){\footnotesize{$\mathbf{q}$}}
\put(21,62){\footnotesize{$\mathbf{r}$}}
\put(36,62){\footnotesize{$\mathbf{s}$}}
\put(36,47){\footnotesize{$\mathbf{t}$}}
\put(51,47){\footnotesize{$\mathbf{u}$}}
\put(51,32){\footnotesize{$\mathbf{v}$}}

\put(36,55){\footnotesize{$\vdots$}}
\put(21,70){\footnotesize{$\vdots$}}
\put(51,40){\footnotesize{$\vdots$}}
\put(41,50.5){\footnotesize{$\cdots$}}
\put(55.5,36){\footnotesize{$\cdots$}}
\put(25.5,66){\footnotesize{$\cdots$}}

\psline{->}(30,14)(45,14)
\put(35,10){$\rm{T}_1$}
\psline{->}(2,45)(2,60)
\put(-3,52){$\rm{T}_2$}




\psline{->}(105,14)(120,14)
\put(110,10){$\rm{T}_1$}

\pspolygon*[linecolor=lightgreen](95,84)(95,65)(143,65)(143,84)(95,84)
\pspolygon*[linecolor=lightgreen](110,84)(110,50)(143,50)(143,84)(110,84)
\pspolygon*[linecolor=lightgreen](125,84)(125,35)(143,35)(143,84)(125,84)
\pspolygon*[linecolor=lightgreen](140,35)(140,20)(143,20)(143,35)(140,35)

\psline{->}(80,20)(145,20)
\psline(80,35)(143,35)
\psline(80,50)(143,50)
\psline(80,65)(143,65)
\psline(80,80)(143,80)

\psline{->}(80,20)(80,87)
\psline(95,20)(95,84)
\psline(110,20)(110,84)
\psline(125,20)(125,84)
\psline(140,20)(140,84)

\psline[linecolor=darkgreen, linewidth=2.5pt](95,84)(95,65)
\psline[linecolor=darkgreen, linewidth=2.5pt](95,65)(110,65)
\psline[linecolor=darkgreen, linewidth=2.5pt](110,65)(110,50)
\psline[linecolor=darkgreen, linewidth=2.5pt](110,50)(125,50)
\psline[linecolor=darkgreen, linewidth=2.5pt](125,50)(125,35)
\psline[linecolor=darkgreen, linewidth=2.5pt](125,35)(140,35)
\psline[linecolor=darkgreen, linewidth=2.5pt](140,35)(140,20)

\put(76,16){\footnotesize{$(0,0)$}}
\put(109,16){\footnotesize{$\cdots$}}
\put(121.5,16){\footnotesize{$(k,0)$}}
\put(139,16){\footnotesize{$\cdots$}}

\put(96,62){\footnotesize{$\mathbf{p}$}}
\put(111,62){\footnotesize{$\mathbf{q}$}}
\put(111,47){\footnotesize{$\mathbf{r}$}}
\put(126,47){\footnotesize{$\mathbf{s}$}}
\put(126,32){\footnotesize{$\mathbf{t}$}}
\put(141,32){\footnotesize{$\mathbf{u}$}}
\put(141,18){\footnotesize{$\mathbf{v}$}}
\put(111,55){\footnotesize{$\vdots$}}
\put(116,50.5){\footnotesize{$\cdots$}}
\put(101,65.5){\footnotesize{$\cdots$}}
\put(131,35.5){\footnotesize{$\cdots$}}
\put(126,40){\footnotesize{$\vdots$}}
\put(96,70){\footnotesize{$\vdots$}}

}

\end{picture}
\end{center}
\caption{Weight diagrams of the $2$-variable weighted shifts for Type I (left) and Type II (right).}
\label{Figure 4}
\end{figure}

\setlength{\unitlength}{1mm} \psset{unit=1mm}
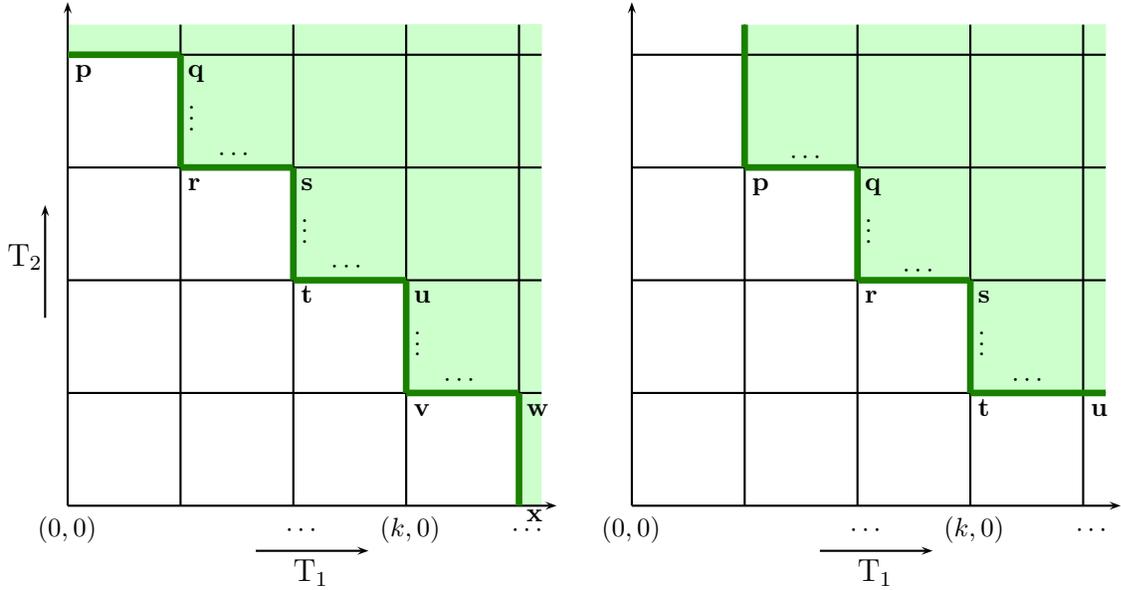
\begin{figure}[th]
\begin{center}
\begin{picture}(150,80)

\rput(0,26){

\pspolygon*[linecolor=lightgreen](5,84)(5,80)(68,80)(68,84)(5,84)
\pspolygon*[linecolor=lightgreen](20,84)(20,65)(68,65)(68,84)(20,84)
\pspolygon*[linecolor=lightgreen](35,84)(35,50)(68,50)(68,84)(35,84)
\pspolygon*[linecolor=lightgreen](50,84)(50,35)(68,35)(68,84)(50,84)
\pspolygon*[linecolor=lightgreen](65,84)(65,20)(68,20)(68,84)(65,84)

\psline{->}(5,20)(70,20)
\psline(5,35)(68,35)
\psline(5,50)(68,50)
\psline(5,65)(68,65)
\psline(5,80)(68,80)

\psline{->}(5,20)(5,87)
\psline(20,20)(20,84)
\psline(35,20)(35,84)
\psline(50,20)(50,84)
\psline(65,20)(65,84)

\psline[linecolor=darkgreen, linewidth=2.5pt](5,80)(20,80)
\psline[linecolor=darkgreen, linewidth=2.5pt](20,80)(20,65)
\psline[linecolor=darkgreen, linewidth=2.5pt](20,65)(35,65)
\psline[linecolor=darkgreen, linewidth=2.5pt](35,65)(35,50)
\psline[linecolor=darkgreen, linewidth=2.5pt](35,50)(50,50)
\psline[linecolor=darkgreen, linewidth=2.5pt](50,50)(50,35)
\psline[linecolor=darkgreen, linewidth=2.5pt](50,35)(65,35)
\psline[linecolor=darkgreen, linewidth=2.5pt](65,35)(65,20)

\put(1,16){\footnotesize{$(0,0)$}}
\put(34,16){\footnotesize{$\cdots$}}
\put(46.5,16){\footnotesize{$(k,0)$}}
\put(64,16){\footnotesize{$\cdots$}}

\put(6,77){\footnotesize{$\mathbf{p}$}}
\put(21,77){\footnotesize{$\mathbf{q}$}}
\put(21,62){\footnotesize{$\mathbf{r}$}}
\put(36,62){\footnotesize{$\mathbf{s}$}}
\put(36,47){\footnotesize{$\mathbf{t}$}}
\put(51,47){\footnotesize{$\mathbf{u}$}}
\put(51,32){\footnotesize{$\mathbf{v}$}}
\put(66,32){\footnotesize{$\mathbf{w}$}}
\put(66,18){\footnotesize{$\mathbf{x}$}}

\put(21,70){\footnotesize{$\vdots$}}
\put(36,55){\footnotesize{$\vdots$}}
\put(51,40){\footnotesize{$\vdots$}}
\put(55,36){\footnotesize{$\cdots$}}
\put(40,51){\footnotesize{$\cdots$}}
\put(25,66){\footnotesize{$\cdots$}}

\psline{->}(30,14)(45,14)
\put(35,10){$\rm{T}_1$}
\psline{->}(2,45)(2,60)
\put(-3,52){$\rm{T}_2$}




\psline{->}(105,14)(120,14)
\put(110,10){$\rm{T}_1$}

\pspolygon*[linecolor=lightgreen](95,84)(95,65)(143,65)(143,84)(95,84)
\pspolygon*[linecolor=lightgreen](110,84)(110,50)(143,50)(143,84)(110,84)
\pspolygon*[linecolor=lightgreen](125,84)(125,35)(143,35)(143,84)(125,84)

\psline{->}(80,20)(145,20)
\psline(80,35)(143,35)
\psline(80,50)(143,50)
\psline(80,65)(143,65)
\psline(80,80)(143,80)

\psline{->}(80,20)(80,87)
\psline(95,20)(95,84)
\psline(110,20)(110,84)
\psline(125,20)(125,84)
\psline(140,20)(140,84)

\psline[linecolor=darkgreen, linewidth=2.5pt](95,84)(95,65)
\psline[linecolor=darkgreen, linewidth=2.5pt](95,65)(110,65)
\psline[linecolor=darkgreen, linewidth=2.5pt](110,65)(110,50)
\psline[linecolor=darkgreen, linewidth=2.5pt](110,50)(125,50)
\psline[linecolor=darkgreen, linewidth=2.5pt](125,50)(125,35)
\psline[linecolor=darkgreen, linewidth=2.5pt](125,35)(143,35)

\put(76,16){\footnotesize{$(0,0)$}}
\put(109,16){\footnotesize{$\cdots$}}
\put(121.5,16){\footnotesize{$(k,0)$}}
\put(139,16){\footnotesize{$\cdots$}}

\put(96,62){\footnotesize{$\mathbf{p}$}}
\put(111,62){\footnotesize{$\mathbf{q}$}}
\put(111,47){\footnotesize{$\mathbf{r}$}}
\put(126,47){\footnotesize{$\mathbf{s}$}}
\put(126,32){\footnotesize{$\mathbf{t}$}}
\put(141,32){\footnotesize{$\mathbf{u}$}}
\put(111,55){\footnotesize{$\vdots$}}
\put(126,40){\footnotesize{$\vdots$}}
\put(116,50.5){\footnotesize{$\cdots$}}
\put(101,65.5){\footnotesize{$\cdots$}}
\put(130.5,36){\footnotesize{$\cdots$}}

}

\end{picture}
\end{center}
\caption{Weight diagrams of the $2$-variable weighted shifts for Type III (left) and Type IV (right).}
\label{Figure 5}
\end{figure}

\medskip
{\bf Type I}: \ $\stc(P) \bigcap (0 \times \mathbb{Z}_+) \ne \emptyset$ and $\stc(P) \bigcap (\mathbb{Z}_+ \times 0) = \emptyset$ . \smallskip \newline 
We refer the reader to the left staircase in Figure \ref{Figure 4}, and recall that we denote arbitrary points in $\Z$ as $\bk \equiv (k_1,k_2)$. \ Observe that the points $\mathbf{p}, \mathbf{q}, \mathbf{r}$ determine a weight diagram as in the generalized one-step extension case of ROMP (Theorem \ref{thmgeneralized}). \ As a result, assuming that the natural necessary conditions for solubility are satisfied, we can extend $\nu_{\mathbf{pr}}$ to the subspace $\mathcal{L}_{(p_1,r_2)}$. \ We therefore reduce the original ROMP to a new ROMP whose descending staircase starts at $(p_1,r_2)$, and continues with the points $\mathbf{s}$, $\mathbf{t}$, $\mathbf{u}$ and $\mathbf{v}$. \ The situation is then almost identical to what we had before, so with the natural necessary conditions for solubility, and using Theorem \ref{thmgeneralized} once again, we extend the measure to the subspace $\mathcal{L}_{(p_1,t_2)}$. \ This yields a new descending staircase, connecting $(p_1,t_2)$ to $\mathbf{u}$ and $\mathbf{v}$. \ Another application of Theorem \ref{thmgeneralized} leads to an extension to the subspace $\mathcal{L}_{(p_1,v_2)}$. \ To finish, we now use the back-step extension result (Lemma \ref{backext}) applied $p_2$ times.

\medskip
{\bf Type II}: \ $\stc(P) \bigcap (0 \times \mathbb{Z}_+) = \emptyset$ and $\stc(P) \bigcap (\mathbb{Z}_+ \times 0) \ne \emptyset$ . \smallskip \newline
This case is completely analogous to the previous one, so after successive instances of Theorem \ref{thmgeneralized} we end up with the subspace $\mathcal{L}_{(p_1,0)}$, and we then apply Lemma \ref{backext}, $p_1$ times, to obtain the Berger measure $\mu$ of $\W$.

\medskip
{\bf Type III}: \ $\stc(P) \bigcap (0 \times \mathbb{Z}_+) \ne \emptyset$ and $\stc(P) \bigcap (\mathbb{Z}_+ \times 0) \ne \emptyset$ . \smallskip \newline
Here repeated application of Theorem \ref{thmgeneralized} does the job.

\medskip
{\bf Type IV}: \ $\stc(P) \bigcap (0 \times \mathbb{Z}_+) = \emptyset$ and $\stc(P) \bigcap (\mathbb{Z}_+ \times 0) = \emptyset$ . \medskip\ \newline
In this case, and after a few instances of Theorem \ref{thmgeneralized}, we end up with the subspace $\mathcal{L}_{(p_1,t_2)}$. \ This case fits well within the scope of Theorem \ref{bigthm}. \ As a result, the natural necessary conditions are sufficient for the existence of the Berger measure $\mu$.

The proof is now complete.
\end{proof}


\end{document}